\documentclass[a4paper,10pt]{amsart}
\usepackage{graphicx}
\usepackage{amssymb,amsmath,amstext,amscd}

\theoremstyle{plain}
\newtheorem{thm}{Theorem}[section]
\newtheorem{proposition}{Proposition}[section]
\newtheorem{lemma}{Lemma}[section]
\newtheorem*{corollary}{Corollary}
\newtheorem*{main theorem}{Theorem}

\theoremstyle{definition}

\begin{document}

\title{Causal and conformal structures of globally hyperbolic spacetimes}
\author{Do-Hyung Kim}
\address{Department of Mathematics, College of Natural Science, Dankook University,
San 29, Anseo-dong, Dongnam-gu, Cheonan-si, Chungnam, 330-714,
Republic of Korea} \email{mathph@dankook.ac.kr}


\begin{abstract}
The group of conformal diffeomorphisms and the group of causal
automorphisms on two-dimensional globally hyperbolic spacetimes
are clarified. It is shown that if spacetimes have non-compact
Cauchy surfaces, then the groups are subgroups of that of
two-dimensional Minkowski spacetime, and if spacetimes have
compact Cauchy surfaces, then the groups are subgroups of that of
two-dimensional Einstein's static universe.
\end{abstract}

\maketitle

\section{Introduction} \label{section:1}

Liouville's Theorem states that there are some kind of rigidity on
 conformal structures of semi-Euclidean space $\mathbb{R}^n_\nu$
when $n \geq 3$. In other words, any conformal diffeomorphisms
defined on an open subset $U$ of $\mathbb{R}^n_\nu$ are generated
by homotheties, isometries and inversions.(\cite{DFN},
\cite{Blair}, \cite{Hartman}, \cite{Lehto}) Since inversion has
singularity, to study conformal structure, we need conformal
compactifications. (\cite{Angles}, \cite{Schott})

In contrast to this, in two-dimensional Euclidean space, it is
known that any conformal diffeomorphisms defined on an open subset
of $\mathbb{R}^2$ are homography or anti-homography and these can
be seen as a conformal map defined on Riemann sphere.
(\cite{Angles})

In this paper, causal structures and conformal structures of
two-dimensional globally hyperbolic spacetimes are analyzed.
Though some authors introduce conformal compactification of
two-dimensional Minkowski spacetime, if we confine the subject to
spacetimes with Cauchy surfaces, we can explicitly obtain their
groups of conformal diffeomorphisms without compactifications. It
is known that the group of conformal diffeomorphisms can be
obtained by the group of causal automorphisms if the dimension of
the Lorentzian manifold is bigger than two (\cite{Fullwood},
\cite{HKM}, \cite{Malament})and so, in high dimensional Lorentzian
manifolds to study conformal structures is equivalent to study
causal structures. However, this is not the case for
two-dimensional spacetimes.

For this reason, to study conformal or causal structure of
two-dimensional spacetimes has sufficient meanings and so, in this
paper, we study coherently both causal and conformal structures of
two-dimensional spacetimes with Cauchy surfaces by tools developed
in Section 3. One of the main results is that if two-dimensional
spacetimes have non-compact Cauchy surfaces, then their structure
groups are subgroups of that of two-dimensional Minkowski
spacetime $\mathbb{R}^2_1$, and if two-dimensional spacetimes have
compact Cauchy surfaces, then their structure groups are subgroups
of that of two-dimensional Einstein's static universe $E$. In this
sense, $\mathbb{R}^2_1$ and $E$ play the role of co-universal
objects among two-dimensional globally hyperbolic spacetimes.

\section{Preliminaries} \label{section:2}

A Lorentzian manifold is a differentiable manifold with the
signature of metric as $(-,+, \cdots, +)$. A tangent vector $v \in
T_pM$ is called timelike, null and spacelike if $g(v,v)$ is less
than 0, equal to 0 and greater than 0, respectively. We say that a
tangent vector is causal if it is timelike or null. It is easy to
see that the set of all causal vectors has two connected
components and we choose one of them to be future-directed vectors
and the other to be past-directed vectors. It is a well-known fact
that a differentiable manifold $M$ has a Lorentzian metric if and
only if $M$ has a nowhere-vanishing vector field $X$. This
nowhere-vanishing vector field can be used to define a
time-orientation which determined future-directed vectors. By
spacetime, we mean a Lorentzian manifold with time-orientation.

We denote by $x \leq y$ if there exists a continuous curve
$\gamma$ from $x$ to $y$ such that for each $t$, there exists a
neighborhood $U$ of $\gamma(t)$ such that $\gamma(t_1) \leq
\gamma(t_2)$ for $t_1 < t < t_2$ and $\gamma(t_i) \in U$. When $x
\leq y$ we say that $x$ and $y$ are causally related or $y$ lies
in the future of $x$. By use of convex normal neighborhood, it can
be shown that $x \leq y$ if and only if there exists a piecewise
differentiable curve $\gamma$ such that $\gamma^\prime(t)$ is
future-directed and causal for each $t$.

A bijective map $f : M \rightarrow N$ between two Lorentzian
manifolds is called a causal isomorphism (anti-causal isomorphism,
respectively) if $f$ satisfied the condition that $x \leq y$ if
and only if $f(x) \leq f(y)$ ($f(x) \geq f(y)$, respectively.)
When the domain of definition and the codomain coincides, we call
the causal isomorphism as a causal automorphism.

It turns out that causal relation of a Lorentzian manifold has
close relations to conformal structure of the given manifold. In
1964, Zeeman has shown that any causal isomorphism on
$n$-dimensional Minkowski spacetime $\mathbb{R}^n_1$ is generated
by homothety and isometries if $n \geq 3$.(\cite{Zeeman}) In 1976,
Hawking et al. had shown that if a spacetime is strongly causal,
causal isomorphism becomes a smooth conformal
diffeomorphism.(\cite{HKM}) In 1977, Malament had shown that
causal isomorphism on any spacetime is a smooth conformal
diffeomorphism.(\cite{Malament}) However, as authors commented,
their results do not hold for $n=2$. In two-dimensional Minkowski
spacetime, there are many more continuous causal
isomorphisms.(\cite{CQG3}, \cite{CQG4}).

\section{Causal structure and covering space} \label{section:3}

Given a covering map $\pi : \overline{M} \rightarrow M$ where $M$
is a semi-Riemannian manifold with metric $g$, we define the
metric of $\overline{M}$ by use of pull-back $\overline{g} = \pi^*
g$. Then $\pi$ is a smooth local isometry. When $M$ is a
Lorentzian manifold, we define a time-orientation on
$\overline{M}$ in such a way that $\pi$ is a time-orientation
preserving local isometry. To be precise, if a vector field $X^a$
defines a time-orientation on $M$, then the pull-back 1-form
$\pi^*X_a$ can be used to define the time-orientation on
$\overline{M}$.

\begin{thm} \label{iso}

Let $\pi_M : \overline{M} \rightarrow M$ and $\pi_N : \overline{N}
\rightarrow N$ be universal covering maps of spacetimes. If $f : M
\rightarrow N$ is a causal isomorphism, then any lift of $f \circ
\pi_M$ through $\pi_N$ is a causal isomorphism.

\end{thm}

\begin{proof}
 Choose $x \in M$ and
$\overline{x} \in \pi_M^{-1}(x)$. Let $y=f(x)$ and choose
$\overline{y} \in \pi_N^{-1}(y)$.

Since $\overline{M}$ is simply-connected, we can lift $f \circ
\pi_M$ through $\pi_N$ and so we get a map $\overline{f} :
(\overline{M}, \overline{x}) \rightarrow (\overline{N}, \overline{
y})$ that satisfy $\pi_N \circ \overline{f} = f \circ \pi_M$.

Since $\overline{N}$ is simply-connected, we can lift $f^{-1}
\circ \pi_N$ through $\pi_M$ and so we get a map
$\overline{f^{-1}} : (\overline{N}, \overline{y}) \rightarrow
(\overline{M}, \overline{x})$ that satisfy $\pi_M \circ
\overline{f^{-1}} = f^{-1} \circ \pi_N$.

By combining the above two equalities, we have $\pi_M \circ
\overline{f^{-1}} \circ \overline{f} = \pi_M$ and $\pi_N \circ
\overline{f} \circ \overline{f^{-1}} = \pi_N$.

Since $\overline{f^{-1}} \circ \overline{f} (\overline{x}) =
\overline{x}$ and $\overline{f} \circ
\overline{f^{-1}}(\overline{y}) = \overline{y}$, by the uniqueness
of lifts, we must have $\overline{f^{-1}} \circ \overline{f} =
Id_{\overline{M}}$ and $\overline{f} \circ \overline{f^{-1}} =
Id_{\overline{N}}$. Therefore, $\overline{f}$ is a bijection from
$\overline{M}$ to $\overline{N}$ with its inverse
$(\overline{f})^{-1} = \overline{f^{-1}}$.

We now show that $\overline{f}$ is a causal isomorphism. Choose
$\overline{x_1}$ and $\overline{x_2}$ in $\overline{M}$ such that
$\overline{x_1} \leq \overline{x_2}$ and let $\overline{\alpha}$
be a future-directed causal curve from $\overline{x_1}$ to
$\overline{x_2}$. Then, since $\pi_M$ is a time-orientation
preserving local isometry and $f$ is a causal isomorphism, the
curve $f \circ \pi_M \circ \overline{\alpha}$ is a future-directed
causal curve in $N$.

Since $\pi_N \circ f = f \circ \pi_M$, $\overline{f} \circ
\overline{\alpha}$ is the lift of the causal curve $f \circ \pi_M
\circ \overline{\alpha}$ through $\pi_N$. Since $\pi_N$ is a
time-orientation preserving local isometry, $\overline{f} \circ
\overline{\alpha}$ is a future-directed causal curve and thus we
have $\overline{f}(\overline{x_1}) \leq
\overline{f}(\overline{x_2})$.

By applying the same argument for $\overline{f^{-1}}$ and
$f^{-1}$, we can show that $\overline{x_1} \leq \overline{x_2}$ if
and only if $\overline{f}(\overline{x_1}) \leq
\overline{f}(\overline{x_2})$ and so $\overline{f}$ is a causal
isomorphism.

\end{proof}

 Under some mild conditions, any causal isomorphisms between two
 Lorentzian manifolds are smooth conformal
 diffeomorphisms.(\cite{HKM}, \cite{Malament}, \cite{Fullwood}). However,
 this is not the case when the dimension of Lorentzian manifold is
 two, and thus we need to prove the previous theorem in topological terms.
  (\cite{Zeeman}, \cite{CQG3}, \cite{CQG4}))

On the other hand, if we assume sufficient smoothness, we can
prove the following.

\begin{thm}
Let $M$ and $N$ be semi-Riemannian manifolds with arbitrary
signatures with universal covering maps $\pi_M : \overline{M}
\rightarrow M$ and $\pi_N : \overline{N} \rightarrow N$. If $f : M
\rightarrow N$ is a conformal diffeomorphism, then any lift of $f
\circ \pi_M$ through $\pi_N$ is a conformal diffeomorphism.

\end{thm}

\begin{proof}
Since $\pi_M$ and $\pi_N$ are smooth local isometry, the same
argument as in the previous theorem applies to this case.
\end{proof}

In the above proofs, though $\overline{f}$ depends on the choice
of $x$, $\overline{x}$ and $\overline{y}$, the proof tells us that
any lift of $f \circ \pi_M$ is a causal isomorphism $\overline{f}
: \overline{M} \rightarrow \overline{N}$.

\begin{proposition}
1. Let $\pi : \overline{M} \rightarrow M$ be a universal covering
of Lorentzian manifold and $Aut(\overline{M})$ be the group of
causal automorphisms of $\overline{M}$. Then, $\pi_1(M)$ is a
subgroup of $Aut(\overline{M})$.\\
2. Let $\pi : \overline{M} \rightarrow M$ be a universal covering
of semi-Riemannian manifold and $Con(\overline{M})$ be the group
of conformal diffeomorphisms of $\overline{M}$. Then, $\pi_1(M)$
is a subgroup of $Con(\overline{M})$.
\end{proposition}

\begin{proof}
 Let $D$ be the group of covering transformations of $\pi : \overline{M} \rightarrow M$, then $\pi_1(M)
= D$ since $\overline{M}$ is simply-connected. It is sufficient to
show that each covering transformation is a causal isomorphism.

If we let $M = N$ and $f = Id_M$ in the previous theorem, then
$\overline{f}$ is a covering transformation of $\pi : \overline{M}
\rightarrow M$. Furthermore, any covering transformation can be
obtained in this way, since any covering transformation is a lift
of $Id_M : M \rightarrow M$. In other words, any covering
transformation is a causal isomorphism from
$\overline{M}$ to $\overline{M}$.\\

 The same argument can be applied to prove the second statement.
\end{proof}

In the following, we are mainly interested in Lorentzian manifold
but since the same argument can be applied to semi-Riemannian
manifold with arbitrary signature, we state the semi-Riemannian
case without proof.

Let $f : M \rightarrow M$ be a causal isomorphism (or, conformal
diffeomorphism, respectively) and $\overline{f} : \overline{M}
\rightarrow \overline{M}$ be a lift of $f \circ \pi$ where
$\overline{M}$ is a universal covering space. Then, for any
covering transformation $\phi \in D$, we have $\pi \circ
\overline{f} \circ \phi = f \circ \pi \circ \phi = f \circ \pi$.
In other words, $\overline{f} \circ \phi \in Aut(\overline{M})$
 \, (or, $\overline{f} \circ \phi \in Con(\overline{M})$, respectively)
is a lift of $f \circ \pi$. Since $\overline{M}$ is
simply-connected, $D$ acts on each fiber $\pi^{-1}(x)$
transitively. Therefore, any lift of $f \circ \pi$ is in $\{
\overline{f} \circ \phi \,\, | \,\, \phi \in D \}$.

\begin{thm} \label{lift}
Let $\overline{f} : \overline{M} \rightarrow \overline{M}$ be a
lift of $f \circ \pi$. Then, for any lift $\overline{g} :
\overline{M} \rightarrow \overline{M}$ of $f \circ \pi$, there
exists $\phi$ and $\psi$ in $D$ such that $\overline{g} =
\overline{f} \circ \phi$ and $\overline{g} = \psi \circ
\overline{f}$.
\end{thm}

\begin{proof}

Choose $\overline{x} \in \pi^{-1}(x)$ for some $x$ and let
$\overline{y}$ be such that $\overline{g}(\overline{y}) =
\overline{x}$ and $\overline{z}$ be such that
$\overline{f}(\overline{z}) = \overline{x}$. Since $D$ acts on
$\pi^{-1}(f^{-1}(x))$ transitively, there exists $\phi \in D$ such
that $\phi(\overline{y})=\overline{z}$. Then, $\overline{f} \circ
\phi (\overline{y}) = \overline{f}(\overline{z}) = \overline{x}$.
By the uniqueness of lifts, we have $\overline{f} \circ \phi =
\overline{g}$.

Likewise, since $D$ acts on $\pi^{-1}(f^{-1}(x))$ transitively,
there exists $\psi \in D$ such that
$\overline{\psi}(\overline{f}(\overline{x})) =
\overline{g}(\overline{x})$. Since $\psi \circ \overline{f}$ is a
lift of $f \circ \pi$, by uniqueness of lifts, we have $\psi \circ
\overline{f} = \overline{g}$.

\end{proof}

In the Theorem \ref{iso}, we have shown that any causal
isomorphism $f : M \rightarrow M$ can be lifted to $\overline{f} :
\overline{M} \rightarrow \overline{M}$. However, in general, it is
not true that, given causal isomorphism $\overline{f} :
\overline{M} \rightarrow \overline{M}$ there exist a causal
isomorphism $f : M \rightarrow M$ such that $f \circ \pi = \pi
\circ \overline{f}$. In the following, we study causal isomorphism
$\overline{f}$ on $\overline{M}$ which have causal isomorphism $f$
on $M$ such that $f \circ \pi = \pi \circ \overline{f}$.

\begin{proposition} \label{subgroup}
Let $A$ be the set of all causal isomorphisms (or, conformal
diffeomorphisms, respectively) $\phi : \overline{M} \rightarrow
\overline{M}$ such that $\pi \circ \phi = f \circ \pi$ for some
causal isomorphism (or, conformal diffeomorphism, respectively) $f
: M \rightarrow M$. Then $A$ is a subgroup of $Aut(\overline{M})$
($Con(\overline{M})$, respectively).
\end{proposition}

\begin{proof}
Choose $\phi$ and $\psi$ in $A$. Then there exist $f_1$ and $f_2$
such that $\pi \circ \phi = f_1 \circ \pi$ and $\pi \circ \psi =
f_2 \circ \pi$ and thus we have $\pi \circ \phi \circ \psi = f_1
\circ \pi \circ \psi = f_1 \circ f_2 \circ \pi$. Therefore, we
have $\phi \circ \psi \in A$. Also, if $\pi \circ \phi = f \circ
\pi$, then $f^{-1} \circ \pi = \pi \circ \phi^{-1}$ and thus
$\phi^{-1} \in A$.
\end{proof}

\begin{proposition} \label{normal}
$D$ is a normal subgroup of  $A$.
\end{proposition}

\begin{proof}
 Let $\overline{\alpha} \in A$ and $\phi
\in D$. It suffices to show $\overline{\alpha} \circ \phi \circ
\overline{\alpha}^{-1} \in D$. We have

\begin{eqnarray*}
\pi \circ \overline{\alpha} \circ \phi \circ
\overline{\alpha}^{-1} &=& \alpha \circ \pi \circ \phi \circ
(\overline{\alpha})^{-1} \,\,\,\,\,  (\because \pi \circ
\overline{\alpha} =
\alpha \circ \pi))\\
&=& \alpha \circ \pi \circ (\overline{\alpha})^{-1} \,\,\,\,\,
(\because \pi
\circ \phi = \pi)\\
&=& \alpha \circ \pi \circ \overline{\alpha^{-1}} \\
&=& \alpha \circ \alpha^{-1} \circ \pi \,\,\,\,\,  (\because \pi
\circ
\overline{f} = f \circ \pi)\\
&=& \pi \\
\end{eqnarray*}

Therefore, we have $\overline{\alpha} \circ \phi \circ
\overline{\alpha}^{-1} \in D$.

\end{proof}

From the above proposition, we can see that if the covering space
$\overline{M}$ is universal, then, since $\pi_1(M) = D$, $A$ is
not trivial if $\pi_1(M)$ is non-trivial.
 Furthermore, we can show
that $A$ is the maximal subgroup of $Aut(\overline{M})$ that
contains $D$ as a normal subgroup. In other words, $A$ is the
normalizer $N(D)$ of $D$ in $Aut(\overline{M})$.

\begin{thm} \label{all}
Let $\pi : \overline{M} \rightarrow M$ be a universal covering
space and $G$ be a subgroup of $Aut(\overline{M})$ (or,
$Con(\overline{M})$, respectively) that contains $D$ as a normal
subgroup. Then $G$ is a subgroup of $A$.
\end{thm}
\begin{proof}
Let $\overline{f} : \overline{M} \rightarrow \overline{M}$ be in
$G$. It suffices to find a causal isomorphism $f : M \rightarrow
M$ such that $\pi \circ \overline{f} = f \circ \pi$. Define $f : M
\rightarrow M$ as follows. For $x \in M$, choose $\overline{x} \in
\pi^{-1}(x)$ and let $f(x) = \pi \circ
\overline{f}(\overline{x})$. To show $f$ is well-defined, let
$\pi(\overline{x}_1) = \pi(\overline{x}_2) = x$. Since
$\overline{M}$ is a universal covering, $D$ acts transitively and
thus we can choose $\varphi \in D$ such that
$\varphi(\overline{x}_1) = \overline{x}_2$ and then we have
$\overline{f} \circ \varphi(\overline{x}_1) =
\overline{f}(\overline{x}_2)$. Since $D$ is normal in $G$, we can
find $\psi \in D$ such that $\psi \circ
\overline{f}(\overline{x}_1) = \overline{f}(\overline{x}_2)$.
Therefore, we have $\pi \circ \psi \circ \circ
\overline{f}(\overline{x}_1) = \pi \circ
\overline{f}(\overline{x}_2)$ and so $\pi \circ \overline{f}
(\overline{x}_1) = \pi \circ \overline{f}(\overline{x}_2)$.

Obviously, $f$ is surjective and we now show that $f$ is
injective. If $f(x_1) = f(x_2)$, then we have $\pi \circ
\overline{f}(\overline{x}_1) = \pi \circ
\overline{f}(\overline{x}_2)$ where $\pi(\overline{x}_1) = x_1$
and $\pi(\overline{x}_2) = x_2$. In other words,
$\overline{f}(\overline{x}_1) = \overline{f}(\overline{x}_2)$ lie
in the same fiber. Thus, there exists $\varphi \in D$ such that
$\varphi \circ \overline{f}(\overline{x}_1) =
\overline{f}(\overline{x}_2)$ and so there exists $\psi \in D$
such that $\overline{f} \circ \psi(\overline{x}_1) =
\overline{f}(\overline{x}_2)$ since $D$ is normal in $G$. Since
$\overline{f}$ is injective, we have $\psi(\overline{x}_1) =
\overline{x}_2$. Therefore, we have $\pi \circ
\psi(\overline{x}_1) = \pi(\overline{x}_2)$ and thus we have $x_1
= \pi(\overline{x}_1) = \pi(\overline{x}_2) = x_2$.

\end{proof}

We now prove one of main theorems.

\begin{thm} \label{main}
Let $M$ be a Lorentzian manifold (or, semi-Riemannian manifold)
with universal covering $\pi : \overline{M} \rightarrow M$. Then,
$Aut(M)$ (or, $Con(M)$) is isomorphic to $N(D)/D$ in which $N(D)$
is the normalizer of $D$ in $Aut(\overline{M})$ (or,
$Con(\overline{M})$.)
\end{thm}

\begin{proof}

Define $\Phi : Aut(M) \rightarrow A/D$ by $\Phi(f) = \overline{f}
D$ where $A$ is the group defined in Proposition \ref{subgroup}.
Then, by Theorem \ref{lift}, $\Phi$ is well-defined and it is
obvious that $\Phi$ is surjective. If $\overline{f}D =
\overline{g}D$, then there exists $\varphi \in D$ such that
$\overline{f} = \overline{g} \circ \varphi$. Therefore, from $f
\circ \pi = \pi \circ \overline{f} = \pi \circ \overline{g} \circ
\varphi = g \circ \pi \circ \varphi = g \circ \pi$, we have $f =
g$ and thus $\Phi$ is injective.

To show that $\Phi$ is a homomorphism, let $f$ and $g$ be in
$Aut(M)$ and $\overline{f}$ and $\overline{g}$ be lifts of $f$ and
$g$. Since $\pi \circ \overline{g} \circ \overline{f} = g \circ
\pi \circ \overline{f} = g \circ f \circ \pi$, $\overline{g} \circ
\overline{f}$ is a lift of $g \circ f$ and thus $\Phi(g \circ f) =
\Phi(g)\Phi(f)$.

By Proposition \ref{subgroup}, $A$ is a subgroup of
$Aut(\overline{M})$ and by Theorem \ref{all}, $A$ is the
normalizer of $D$. Therefore, $\Phi$ is an isomorphism from
$Aut(M)$ to $N(G)/G$.

\end{proof}

When $M$ is a globally hyperbolic spacetime with Cauchy surface
$\Sigma$, by Theorem 1 in \cite{Bernal1}, there exists a
diffeomorphism $f : \mathbb{R} \times \Sigma \rightarrow M$. Since
$\mathbb{R} \times \Sigma$ is homotopy equivalent to $\Sigma$, we
have $\pi_1(M) = \pi_1(\Sigma)$. Therefore, from the above
theorem, we can see that causal structure and conformal structure
of $M$ depend on corresponding structure of its universal covering
space and topological structure of its Cauchy surface $\Sigma$. To
make it explicit, let $p : \overline{\Sigma} \rightarrow \Sigma$
be a universal covering map. Then, $\pi : \mathbb{R} \times
\overline{\Sigma} \rightarrow M$ defined by $\pi(t,\overline{x}) =
f(t, p(\overline{x}))$ is a universal covering map of $M$. Then
the above theorem tells us that the group of causal isomorphisms
of $M$ is isomorphic to $N/\pi_1(\Sigma)$ where $N$ is the
normalizer of $D$.

By Theorem 3.69 in \cite{Beem}, for any given complete Riemannian
manifold $\Sigma$, we can make $M = \mathbb{R} \times \Sigma$ into
globally hyperbolic spacetime with Cauchy surface $\Sigma$. Then,
we can see that the group of causal isomorphisms, $Aut(M)$, is
given by $N/ \pi_1(\Sigma)$.

For example, let $\Sigma$ be a complete Riemannian manifold $T^2 =
S^1 \times S^1$. Then $M = \mathbb{R} \times T^2$ is globally
hyperbolic with its universal covering space diffeomorphic to
$\mathbb{R}^3$. If we use a usual metric on $T^2$ and covering map
$\pi(t,u,v) = (t, e^{iu}, e^{iv})$, then the pull-back metric on
the universal covering space $\overline{M}$ is the Minkowski
metric on $\mathbb{R}^3_1$. Therefore, we can see that
$Aut(\overline{M})$ is isomorphic to $Aut(\mathbb{R}^3_1)$ which
is generated by homothety and isometry, where homothety factor
must be pisitive by Zeeman's theorem.(Ref. \cite{Zeeman}). Then,
Theorem \ref{main} tells us that $Aut(M)$ is isomorphic to $N / (
\mathbb{Z} \times \mathbb{Z} )$ where $N$ is a normalizer of
$\mathbb{Z} \times \mathbb{Z}$ in $Aut(\mathbb{R}^3_1)$. If we
want to get $Con(M)$, since $Con(\mathbb{R}^3_1)$ is given by
Liouville's theorem(\cite{DFN}, \cite{Blair}), we can get the
similar result.

As a second example, we consider the universal covering map $\pi :
S^n \rightarrow \mathbb{R}P^n$ for $n \geq 2$. The group of
covering transformations $D$ consists of the identity map and the
antipodal map. Then, it is easy to see that $g \in N(D)$ if and
only if $g(-x)=-g(x)$. Therefore, we have $Con(\mathbb{R}P^n) = \{
g \in Con(S^n) \,\, | \,\, g(-x)=-g(x) \}/\sim$ where the
equivalence relation $\sim$ is given by $g \sim -g$ for each $g
\in Con(S^n)$.

In fact, for $n \geq 3$, since $Con(S^n)$ is isomorphic to
$SO(n+1,1)$, we have that $Con(\mathbb{R}P^n)$ is isomorphic to
$N/\mathbb{Z}_2$ where $N$ is the normalizer of $\mathbb{Z}_2$ in
$SO(n+1,1)$.

As can be seen in the above examples and Theorem \ref{main}, to
obtain the group of causal automorphisms of a globally hyperbolic
Lorentzian manifold, we only need to study causal structures of
universal covering space and the topological structures of its
Cauchy surface.

\section{Causal structure of two-dimensional spacetimes} \label{sectione:4}

Two-dimensional spacetimes play an important role in string theory
since world sheets of strings can be seen as two-dimensional
spacetimes. Therefore, to study conformal and causal structure of
two-dimensional spacetimes is important. In this section, we study
causal structures two-dimensional globally hyperbolic spacetimes.

Causal structure of two-dimensional spacetimes with non-compact
Cauchy surfaces has been analyzed in \cite{JGP1} and we briefly
review the result since they play a key role in the analysis of
two-dimensional spacetimes with compact Cauchy surfaces.

Let $M$ be a two-dimensional spacetime with non-compact Cauchy
surfaces. Then its Cauchy surface $\Sigma$ is homeomorphic to
$\mathbb{R}$. If we choose a homeomorphism from $\Sigma$ to the
Cauchy surface $\{ \, (x,0) \,\, | \,\, x \in \mathbb{R} \, \}$ of
$\mathbb{R}^2_1$, then the homeomorphism can be uniquely extended
to a map from $M$ into a subset of $\mathbb{R}^2_1$ that contains
$x$-axis as a Cauchy surface in such a way that the extended map
preserves causal relations. The details can be found in
\cite{CQG2}. The following is Theorem 5.1 in \cite{CQG2}.

\begin{thm} \label{imbedding}
Any two-dimensional spacetime with a non-compact Cauchy surface
can be causally isomorphically imbedded in $\mathbb{R}^2_1$.
\end{thm}

The following is Theorem 2.2 in \cite{CQG4}.

\begin{thm} \label{auto}
Let $F : \mathbb{R}^2_1 \rightarrow \mathbb{R}^2_1$ be a causal
automorphisms on $\mathbb{R}^2_1$. Then, there exist unique
homeomorphisms $\varphi$ and $\psi$ of $\mathbb{R}$, which are
either both increasing or both decreasing, such that if $\varphi$
and $\psi$ are increasing, then we have $F(x,t) = \frac{1}{2}
\Big( \varphi(x+t)+\psi(x-t), \varphi(x+t)-\psi(x-t) \Big)$, or if
$\varphi$ and $\psi$ are decreasing, then we have $F(x,t) =
\frac{1}{2} \Big(\varphi(x-t)+\psi(x+t), \varphi(x-t)-\psi(x+t)
\Big)$.\\
Conversely, for any given homeomorphisms $\varphi$ and $\psi$ of
$\mathbb{R}$, which are either both increasing or both decreasing,
the function $F$ defined as above is a causal automorphism on
$\mathbb{R}^2_1$.
\end{thm}

We can improve the above theorem in such a way that any causal
isomorphism from an open subset of $\mathbb{R}^2_1$ that contains
$x-$axis as a Cauchy surface onto another open subset of
$\mathbb{R}^2_1$ that contains $x-$axis as a Cauchy surface can be
represented as the above theorem. This is essentially the same as
Theorem 4.3 in \cite{JGP1} and from Theorem \ref{imbedding}, we
have the following theorem.

\begin{thm} \label{JGP}
Let $M$ be a two-dimensional spacetime with non-compact Cauchy
surfaces. Then, the group of all causal automorphisms of $M$ is
isomorphic to a subgroup of the group of all causal automorphisms
of $\mathbb{R}^2_1$.
\end{thm}
\begin{proof}
This is Theorem 4.5 in \cite{JGP1}.
\end{proof}

The proof of Theorem \ref{imbedding} tells us that we can consider
a two-dimensional spacetime $M$ with non-compact Cauchy surfaces
as a subset of $\mathbb{R}^2_1$ which contains the $x-$axis as a
Cauchy surface, and then we only need to study subgroups of
$Aut(\mathbb{R}^2_1)$ since $Aut(M)$ is isomorphic to a subgroup
of $Aut(\mathbb{R}^2_1)$ by Theorem \ref{JGP}.

We must note that, when we consider $Aut(M)$ as a subgroup of
$Aut(\mathbb{R}^2_1)$, the subgroup depends on the choice of
homeomorphism used in Theorem \ref{imbedding}, and so we need to
clarify the dependence of the choice.

\begin{proposition}
Let $M$ be a two-dimensional spacetime with a non-compact Cauchy
surface $\Sigma$. Let $f$ and $g$ be homeomorphisms from $\Sigma$
onto $\{ (x,0) \,\, | \,\, x \in \mathbb{R} \}$ and let $i_f$,
$i_g : M \hookrightarrow \mathbb{R}^2_1$ be causally isomorphic
imbeddings given in Theorem \ref{imbedding}. Then, $Aut(i_f(M))$
and $Aut(i_g(M))$ are conjugate in $Aut(\mathbb{R}^2_1)$
\end{proposition}
\begin{proof}
By the remark preceding Theorem \ref{JGP}, there exists a unique
causal automorphism $F$ on $\mathbb{R}^2_1$ such that $F(i_f(M)) =
i_g(M)$. Then, we have $Aut(i_g(M)) = F Aut(i_f(M)) F^{-1}$.
\end{proof}

In the following, we analyze causal structures of two-dimensional
spacetimes with compact Cauchy surfaces. Two-dimensional de Sitter
spacetime and Einstein static universe fall into this category and
we can apply the argument developed in this section to those
spacetimes and any globally hyperbolic proper subset of those
spacetimes. The causal structures of higher dimensional de Sitter
spacetime and Einstein's static universe are analyzed in
\cite{Lester}. Since de Sitter spacetime and Einstein's static
universe are homeomorphic to each other, the result shows a
characteristic differences between two-dimensional and higher
dimensional spacetimes.

 We now analyze the structure of $A$ when $M$ is a two-dimensional
spacetime with compact Cauchy surfaces since $A$ plays a central
role of causal or conformal structure of $M$. To obtain $A$, since
$A$ is a normalizer of $D$, we can use the group structure of
$Aut(\mathbb{R}^2_1)$ which is given in \cite{CQG4}. However, it
is easy to use the definition of $A$, which is given in
Proposition \ref{subgroup}.

 Since the Cauchy
surface of $M$ is homeomorphic to $S^1$, $M$ is diffeomorphic to
$S^1 \times \mathbb{R}$ by Theorem 1 in \cite{Bernal1} and so we
can use the universal covering $\pi : \overline{M} \rightarrow S^1
\times \mathbb{R}$ defined by $(x,t) \mapsto (e^{2\pi ix}, t)$,
where $\overline{M}$ is diffeomorphic to $\mathbb{R} \times
\mathbb{R}$. We must note that by Theorem \ref{iso},
$\overline{M}$ can be obtained uniquely up to causal isomorphims.
Since $M$ is globally hyperbolic, its universal covering space
$\overline{M}$ is globally hyperbolic with non-compact Cauchy
surfaces by Theorem 14 in \cite{VR} or Theorem 2.1 in \cite{JMP1}.
Then, by Theorem \ref{imbedding}, we can imbed $\overline{M}$ into
$\mathbb{R}^2_1$ causally isomorphically and so we now consider
$\overline{M}$ as a globally hyperbolic subset of $\mathbb{R}^2_1$
which has $x$-axis as a Cauchy surface.

Choose $\overline{g} \in A$. Then by the remark following Theorem
\ref{auto}, there exist homeomorphisms $\varphi$ and $\psi$ on
$\mathbb{R}$, which are either both increasing or both decreasing,
such that, if they are increasing, we have $\overline{g}(x,t) =
\frac{1}{2} \big( \varphi(x+t) + \psi(x-t), \varphi(x+t) -
\psi(x-t) \big)$, and if they are decreasing, we have
$\overline{g}(x,t) = \frac{1}{2} \big( \varphi(x-t) + \psi(x+t),
\varphi(x-t) - \psi(x+t) \big)$. It must be noted that the domain
of $\varphi$ and $\psi$ depend on the structure of $\overline{M}$.

We first analyze the structure of $D$ which is a normal subgroup
of $A$. Let $\Phi \in D$ be given by $(x,t) \mapsto
\big(\alpha(x,t) , \beta(x,t) \big)$. Then, since $\pi \circ \Phi
=\pi$, we must have $\big( e^{2\pi i \alpha}, \beta \big) =
\big(e^{2\pi ix}, t \big)$ and thus $\alpha(x,t) = x+m$ for some
$m \in \mathbb{Z}$ and $\beta(x,t) = t$. Therefore, we have
$\Phi(x,t) = (x+m,t)$ for some $m \in \mathbb{Z}$. In null
coordinates, $u=x+t$ and $v=x-t$, we have $\Phi(u,v)=(u+m, v+m)$
for some $m \in \mathbb{Z}$. This is the way in which $D \approx
\mathbb{Z}$ acts in $Aut(\mathbb{R}^2_1)$ or $A$.

We now analyze the structure of $A$. Given $\overline{g} \in A$,
we assume that $\varphi$ and $\psi$ are both increasing. Then
there exists $g \in Aut(M)$ such that $\pi \circ \overline{g} = g
\circ \pi$ and we have\\
$$g(e^{2\pi ix}, t) = \Big( e^{\pi i \{\varphi(u)+\psi(v)\}},
\frac{1}{2} ( \varphi(u) - \psi(v) ) \Big)$$ where $u = x+t$ and
$v = x-t$ are null coordinates.

Since $\overline{g}$ is in $A$, $g$ must be well-defined and so we
must have that, for any given $n \in \mathbb{Z}$, there exists $m
\in \mathbb{Z}$ such that
\begin{eqnarray*}
 \varphi(u+n) +
\psi(v+n) &=& \varphi(u)+\psi(v)+m, \mbox{\,\, and}\\
\varphi(u+n) - \psi(v+n) &=& \varphi(u) - \psi(v).
\end{eqnarray*}

These two equations are equivalent to
\begin{eqnarray*}
\varphi(u+n)-\varphi(u) &=& \frac{m}{2} \\
\psi(v+n) - \psi(v) &=& \frac{m}{2}.
\end{eqnarray*}

Conversely, it is easy to see that if two homeomorphisms $\varphi$
and $\psi$ satisfy the above two equations, then $\overline{g} \in
A$ can be obtained.

If we apply exactly the same argument to the case in which both
$\varphi$ and $\psi$ are decreasing, we get the same results and
we have the following theorem.

\begin{thm} \label{form}
Let $M$ be a two-dimensional spacetime with compact Cauchy
surfaces and $\pi : \overline{M} \rightarrow M$ be a universal
covering map. Then, we have the following.\\
(1) The group of covering transformation $D$ consists of those
functions $\Phi$ given by $\Phi(u,v) = (u+m, v+m)$ in null
coordinates. The group $A$ consists of pairs of two homeomorphisms
$(\varphi, \psi) \in Aut(\overline{M})$ on $\mathbb{R}$ that
satisfy the condition : for any $n \in \mathbb{Z}$, there exists
$m \in \mathbb{Z}$ such that $f(x+n) -
f(x) = \frac{m}{2}$ for all $x$.\\
(2) The general form of causal automorphism on $M$ is given by
$$g(e^{2\pi ix}, t) = \Big( e^{\pi i \{\varphi(u)+\psi(v)\}},
\frac{1}{2} ( \varphi(u) - \psi(v) ) \Big)$$ where $\varphi$ and
$\psi$ are given from (1).
\end{thm}
\begin{proof}
The first part has been proved and for the second part, we note
that for any $g \in Aut(M)$, we can find $\overline{g} \in
Aut(\overline{M})$ such that $\pi \circ \overline{g} = g \circ
\pi$.
\end{proof}

From the above theorem, we can see that, regardless of their
increasing behavior, for $(\varphi, \psi)$ to be in $A$, it is
necessary and sufficient to satisfy $f(x+n) - f(x) = \frac{m}{2}$.
The increasing behavior of $\varphi$ and $\psi$ can be read from
the above equation. For example, if they are increasing, $n > 0$
implies that $m > 0$. and if they are decreasing $n > 0$ implies
that $m < 0$.

 Therefore, by continuity, it is sufficient
to specify values of $\varphi$ and $\psi$ on a interval $(0, n)$.
For simplicity, if we take $n=1$, then $\varphi$ and $\psi$ are
completely determined by specifying bijections from $(0,1)$ onto
an interval whose length is an half integer. We also remark that
the condition that $f(x+n) - f(x) = \frac{m}{2}$ is equivalent to
the condition that $f(1)-f(0) = \frac{m^\prime}{2}$ and
$f(x+n^\prime)-f(x)=\frac{m^\prime n^\prime}{2}$ for some
$m^\prime \in \mathbb{Z}$ and for all $x \in \mathbb{R}$ and
$n^\prime \in \mathbb{Z}$. It must be also noted that the number
$m^\prime$ is inherited from the causal structure of $M$.

In \cite{Lester}, causal structure of high dimensional de Sitter
spacetime and Einstein's static universe are analyzed and as
expected, they exhibit quite different behaviors from that
two-dimensional spacetimes with compact Cauchy surfaces. We also
remark that, as the above theorem shows, in two-dimensional
spacetimes with compact Cauchy surfaces, the general form of
causal automorphisms of such spacetimes has the same pattern.
Also, their causal automorphism group have similar patterns as the
following theorem shows.

\begin{thm}
Let $M$ be a two-dimensional spacetime with compact Cauchy
surfaces. Then, there exists two-dimensional spacetime
$\overline{M}$ with non-compact Cauchy surfaces of which the group
of causal automorphisms contains $D$ as a subgroup such that
$Aut(M)$ is isomorphic to $A/D$ where $A$ is a normalizer of $D$
in $Aut(\overline{M})$. Conversely, given two-dimensional
spacetime $\overline{M}$ with non-compact Cauchy surfaces, if the
group of causal automorphisms of $\overline{M}$ contains $D$ as a
subgroup, then there exists two-dimensional spacetime $M$ with
compact Cauchy surfaces such that $Aut(M)$ is isomorphic to $A/D$.
\end{thm}

\begin{proof}

Let $\pi : \overline{M} \rightarrow M$ be a universal covering and
define $\Phi : Aut(M) \rightarrow A/D$ by $\Phi(f) = \overline{f}
D$ where $A$ is the group defined in Proposition \ref{subgroup}.
Then, by Proposition \ref{normal}, $A$ is the normalizer of $D$
and, by Theorem \ref{lift}, $\Phi$ is well-defined and it is
obvious that $\Phi$ is surjective. If $\overline{f}D =
\overline{g}D$, then there exists $\varphi \in D$ such that
$\overline{f} = \overline{g} \circ \varphi$. Therefore, from $f
\circ \pi = \pi \circ \overline{f} = \pi \circ \overline{g} \circ
\varphi = g \circ \pi \circ \varphi = g \circ \pi$, we have $f =
g$ and thus $\Phi$ is injective.

To show that $\Phi$ is a homomorphism, let $f$ and $g$ be in
$Aut(M)$ and $\overline{f}$ and $\overline{g}$ be lifts of $f$ and
$g$. Since $\pi \circ \overline{g} \circ \overline{f} = g \circ
\pi \circ \overline{f} = g \circ f \circ \pi$, $\overline{g} \circ
\overline{f}$ is a lift of $g \circ f$ and thus $\Phi(g \circ f) =
\Phi(g)\Phi(f)$. Therefore, $\Phi$ is an isomorphism.

We now prove the converse. Let $\overline{M}$ be a two-dimensional
spacetime with non-compact Cauchy surfaces. Then we can consider
$\overline{M}$ as an open subset of $\mathbb{R}^2_1$ that contains
$x$-axis as a Cauchy surface. Since $D$ acts freely and properly
continuously, we can make $M = \overline{M}/D$ into a spacetime in
such a way that the quotient map $\pi : \overline{M} \rightarrow
M$ is a time-orientation preserving covering map, where since
$\overline{M}$ is diffeomorphic to $\mathbb{R}^2$, $\overline{M}$
is a universal covering space of $M$. Furthermore, since $\Sigma =
\{ (x,0 \,\, | \,\, x \in \mathbb{R} \}$ is a Cauchy surface of
$\overline{M}$, $\pi(\Sigma)$, which is homeomorphic to $S^1$, is
a Cauchy surface of $M$. Then, the same argument as in the first
part shows that $\Phi : Aut(M) \rightarrow A/D$ is an isomorphism.

\end{proof}

Let $E$ be the two-dimensional Einstein's static universe. Then
$E$ is topologically $S^1 \times \mathbb{R}$ with the flat metric
$d\theta^2 - dt^2$ and thus its universal covering space is
$\mathbb{R}^2_1$. In the general form of causal automorphism $g$
given in Theorem \ref{form}, the homeomorphisms $\varphi$ and
$\psi$ are determined by the structure of $\overline{M}$. Since
$\overline{M}$ is a subset of $\mathbb{R}^2_1$, we can see that,
for any two-dimensional spacetime $M$ with compact Cauchy
surfaces, $Aut(M)$ is a subgroup of $Aut(E)$. We can also prove
this by use of Theorem \ref{main} as follows.

\begin{thm}
Let $M$ be a two-dimensional spacetime with compact Cauchy
surfaces. Then, $Aut(M)$ is isomorphic to a subgroup of $Aut(E)$.
\end{thm}

\begin{proof}
Let $N_M$ and $N_E$ be normalizers of $D$ in $Aut(\overline{M})$
and $Aut(\overline{E}) = Aut(\mathbb{R}^2_1)$. Then, since
$Aut(\overline{M})$ is a subgroup of $Aut(\mathbb{R}^2_1)$, $N_M$
is a subgroup of $N_E$. Therefore, the homomorphism $f : N_M
\rightarrow N_E/D$ given by $f(g) = gD$ is well-defined and we
have $Ker(f) = D$. Therefore, the induced homomorphism $N_M/D
\rightarrow N_E/D$ is injective and thus $Aut(M) = N_M/D$ is
isomorphic to a subgroup of $N_E/D = Aut(E)$.
\end{proof}

\section{Conformal structure of two-dimensional spacetimes} \label{sectione:5}

In general, it is well-known that any causal isomorphism between
two Lorentzian manifolds is a smooth conformal diffeomorphism if
the dimension of manifolds are bigger than two. However, this is
not the case when the dimension is two. Even if a causal
automorphism is $C^\infty$, it is not necessarily a conformal
diffeomorphism. For example, if we take $\varphi = \psi = x^3$,
then the function $F$ defined as in Theorem \ref{auto} is
$C^\infty$ and causal automorphism on $\mathbb{R}^2_1$. However,
it is not a conformal diffeomorphism since its inverse is not
differentiable at $(0, 0)$. Therefore, if we want to get a
conformal diffeomorphism on $\mathbb{R}^2_1$ from causal
automorphism we need one more condition and we state the
corresponding result in two-dimensional spacetimes with
non-compact Cauchy surfaces.

\begin{lemma} \label{lemma-smooth}
Let $M$ and $N$ be two-dimensional spacetimes with non-compact
Cauchy surfaces and $F : M \rightarrow N$ be a causal isomorphism.
If both $F$ and $F^{-1}$ are $C^\infty$, then, $F$ is a $C^\infty$
conformal diffeomorphism.
\end{lemma}
\begin{proof}
It suffices to show that $F_*$ sends null vectors to null vectors,
by Lemma 2.1 in \cite{Beem}. Let $v \in T_pM$ be a null vector and
let $\gamma$ be a future-directed null geodesic with $\gamma(0)=p$
and $\gamma^\prime(0) = v$. Then, since $M$ has non-compact Cauchy
surfaces, $\gamma$ has no null cut points, and so, for any $t >
0$, we have $\gamma(0) \leq \gamma(t)$ but not $\gamma(0) <<
\gamma(t)$. Since $F$ is a causal isomorphism, we have
$F(\gamma(0)) \leq F(\gamma(t))$ but not $F(\gamma(0)) <<
F(\gamma(t))$. Therefore, any future-directed causal curve from
$F(\gamma(0))$ to $F(\gamma(t))$ is a null pregeodesic. Since $F$
is a $C^\infty$ causal isomorphism, $F \circ \gamma$ is a
future-directed causal curve and thus $F \circ \gamma$ is a null
pregeodesic. Therefore, $F_*(v)$ is a null vector. Likewise, we
can apply the same argument to $F^{-1}$ to obtain the desired
result.
\end{proof}

Let $F : M \rightarrow N$ be an anti-causal isomorphism. When the
time-orientation of $N$ is given by a vector field $X$, if we
replace the time-orientation of $N$ by $-X$, the map $F$ becomes a
causal isomorphism. Since conformal map is irrelevant to
time-orientations, we have the following.

\begin{corollary}
Let $M$ and $N$ be two-dimensional spacetimes with non-compact
Cauchy surfaces and $F : M \rightarrow N$ be an anti-causal
isomorphism. If both $F$ and $F^{-1}$ are $C^\infty$, then, $F$ is
a $C^\infty$ conformal diffeomorphism.
\end{corollary}

 If $M$ is a two-dimensional spacetime with a non-compact Cauchy
 surface $\Sigma$, then $\Sigma$ is homeomorphic to $\mathbb{R}$.
 If we identify $\mathbb{R}$ and $\mathbb{R}_0 = \{ (x,0) \,\, |
 \,\, x \in \mathbb{R}\}$ which is a Cauchy surface of $\mathbb{R}^2_1$,
  we can choose a homeomorphism $f : \Sigma \rightarrow \mathbb{R}_0$.
  For given $p \in J^+(\Sigma)$, let $S_p = J^-(p) \cap \Sigma$. Then, since $M$ is globally
 hyperbolic, $S_p$ is compact and connected and thus $f(S_p)$ is
 also compact and connected subset of $\mathbb{R}_0$ and we can
 choose unique $q \in J^+(\mathbb{R}_0)$ such that $J^-(q) \cap
 \mathbb{R}_0 = f(S_p)$. In this way, we can extend $f$ to a map
 from $J^+(\Sigma)$ into $J^+(\mathbb{R}_0)$. Likewise, we can extend
 $f$ from $J^-(\Sigma)$ into $J^-(\mathbb{R}_0)$ and thus we have a
 map from $M$ into $\mathbb{R}^2_1$. It can be shown that this
 extended map is a causal isomorphism from $M$ into its image in
 $\mathbb{R}^2_1$ and that $\mathbb{R}_0$ is a Cauchy surface
 of the image of the extended map. This is the main idea of Theorem \ref{imbedding} and details
 of the argument can be found in \cite{CQG2}.

 In the above argument, if we take $f$ to be a $C^\infty$ diffeomorphism between
 $\Sigma$ and $\mathbb{R}_0$, then the extended map is a $C^\infty$ conformal
 diffeomorphism.

 \begin{thm} \label{conimbedding}
 Let $M$ be a two-dimensional spacetime with non-compact Cauchy
 surfaces. Then $M$ can be imbedded into $\mathbb{R}^2_1$ in such
 a way that the imbedding is a conformal diffeomorphism onto a
 globally hyperbolic subset of $\mathbb{R}^2_1$ that contains $x$-axis as a
 Cauchy surface.
 \end{thm}
 \begin{proof}
 Let $\Sigma$ be a Cauchy surface of $M$ and take a $C^\infty$ diffeomorphism
 $f : \Sigma \rightarrow \mathbb{R}_0$. Then, by the above
 argument, $f$ can be extended to a causal isomorphism $F : M
 \rightarrow \mathbb{R}^2_1$ onto an open subset that contains $\mathbb{R}_0$ as a Cauchy surface,
  which is a topological imbedding by Theorem \ref{imbedding}.

  By the previous lemma, it is sufficient to show that $F$ and $F^{-1}$ are $C^\infty$.
 For given $p \in J^+(\Sigma)$, since $\Sigma$ is a non-compact smooth
 one-dimensional manifold, $S_p$ is uniquely determined by two boundary
 points, say $x$ and $y$. Since there are two unique null geodesics
 $\gamma_1$ from $x$ to $p$ and $\gamma_2$ from $y$ to $p$, the
 dependence of $p$ on $x$ and $y$ is $C^\infty$. Likewise, the
 dependence of $F(p)$ on $F(x)$ and $F(y)$ is also $C^\infty$.
 Therefore, since $f$ is $C^\infty$, $F$ is $C^\infty$. By
 exactly the same manner, we can show that $F^{-1}$ is $C^\infty$.
 \end{proof}

From the above theorem, we can see that, to analyze conformal
structure of two-dimensional spacetimes with non-compact Cauchy
surfaces, it is sufficient to study conformal structures of an
open subset of $\mathbb{R}^2_1$ that contains $x$-axis as a Cauchy
surface.

\begin{lemma}
Let $U$ be a globally hyperbolic open subset of $\mathbb{R}^2_1$
that contains $x$-axis as a Cauchy surface and let $F : U
\rightarrow \mathbb{R}^2_1$ be a $C^\infty$ conformal
diffeomorphism into an open subset of $\mathbb{R}^2_1$ that
contains $x$-axis. Then, there exists unique $C^\infty$
diffeomorphisms $\varphi$ and $\psi$ of $\mathbb{R}$ such that
$\varphi^\prime \psi^\prime >0$ and $F$ is given by one of the
following form.\\
(1) $F(x,t) = \big( \varphi(x+t)+\psi(x-t), \varphi(x+t)-\psi(x-t)
\big).$\\
(2) $F(x,t) = \big( \varphi(x-t)+\psi(x+t), \varphi(x-t)-\psi(x+t)
\big).$\\
\end{lemma}

\begin{proof}
We only sketch outlines of the proof since it can be obtained from
calculations and simple arguments. If a map $F : (x,t) \mapsto
(X,T)$ is a conformal map, from the definitions of conformal map,
we obtain two cases.
\begin{eqnarray*}
\mbox \,\,\, {(i)} X_x^2&<&T_x^2 \,\,\, \mbox{and} \,\,\, ( X_x = T_t \,\,\, \mbox{or} \,\,\, X_t = T_x).\\
\mbox{(ii)} \,\,\, X_x^2&<&T_x^2 \,\,\, \mbox{and} \,\,\, ( X_x =
-Y_t \,\,\,
\mbox{or} \,\,\, X_t = -T_x).\\
\end{eqnarray*}

We show the case (i) since the case (ii) can be solved by exactly
the same manner.

From $X_x = T_t$ and $X_t = T_x$, we can see that both $X$ and $T$
satisfies wave equation and thus from the general solution of wave
equations in one spatial coordinate, we have $X = \varphi(x+t) +
\psi(x-t)$ and $T = \alpha(x+t)+\beta(x-t)$. From the system of
partial differential equations $X_x = T_t$ and $X_t = T_x$, we
have $X = \varphi(x+t)+\psi(x-t)$ and $T = \varphi(x+t) -
\psi(x-t) +c$ for some $c \in \mathbb{R}$. By replacing $\varphi$
by $\varphi+\frac{c}{2}$ and $\psi$ by $\psi-\frac{c}{2}$, we have
$X = \varphi(x+t)+\psi(x-t)$ and $T = \varphi(x+t)-\psi(x-t)$.
From $X_x^2 < T_x^2$, we obtain $\varphi^\prime \psi^\prime > 0$.
Since the domains of definitions and ranges of $F$ must contain
$x$-axis, $\varphi$ and $\psi$ must be defined on the whole of
$\mathbb{R}$ and their ranges are $\mathbb{R}$. For $F$ to be a
diffeomorphism, $\varphi$ and $\psi$ must be diffeomorphisms.
\end{proof}

We now state a theorem corresponding to Theorem \ref{JGP}.

\begin{thm}
Let $M$ be a two-dimensional spacetime with non-compact Cauchy
surfaces. Then, the group of all conformal diffeomorphisms of $M$
is isomorphic to a subgroup of $Con(\mathbb{R}^2_1)$, the group of
all conformal diffeomorphisms of $\mathbb{R}^2_1$.
\end{thm}

\begin{proof}
By Theorem \ref{conimbedding}, we only need to study the group of
all conformal diffeomorphisms of a globally hyperbolic open subset
$U$ of $\mathbb{R}^2_1$ that contains $x$-axis as a Cauchy
surface. Then, for any conformal diffeomorphism on $U$, by the
previous lemma, we have two unique diffeomorphisms $\varphi$ and
$\psi$ defined on $\mathbb{R}$ in such a way that $F$, defined as
in the previous lemma, is the conformal diffeomorphism of $U$.
Since $\varphi$ and $\psi$ are defined on the whole of
$\mathbb{R}$, we can uniquely extend $F$ to a map $\overline{F}$
 defined on $\mathbb{R}^2_1$ and the extension is a conformal
 diffeomorphism of $\mathbb{R}^2_1$ by the previous lemma. Then, the map $F \mapsto \overline{F}$ is a group
isomorphism from $Con(M)$ into a subgroup of
$Con(\mathbb{R}^2_1)$.
\end{proof}

 This theorem is the counterpart to Theorem \ref{JGP} and the
 argument following Theorem \ref{JGP} can be applied to analysis
 of conformal structure of spacetimes with compact Cauchy surfaces
 and so we state the corresponding theorems without proofs.

 \begin{thm}
Let $M$ be a two-dimensional spacetime with compact Cauchy
surfaces and $\pi : \overline{M} \rightarrow M$ be a universal
covering map. Then, we have the following.\\
(1) The group of covering transformation $D$ consists of those
functions $\Phi$ given by $\Phi(u,v) = (u+m, v+m)$ in null
coordinates. The group $A$, the normalizer of $D$ in $Con(M)$
consists of pairs of two diffeomorphisms $(\varphi, \psi) \in
Aut(\overline{M})$ on $\mathbb{R}$ that satisfy the condition :
for any $n \in \mathbb{Z}$, there exists $m \in \mathbb{Z}$ such
that $f(x+n) -
f(x) = \frac{m}{2}$ for all $x$.\\
(2) The general form of conformal diffeomorphism on $M$ is given
by
$$g(e^{2\pi ix}, t) = \Big( e^{\pi i \{\varphi(u)+\psi(v)\}},
\frac{1}{2} ( \varphi(u) - \psi(v) ) \Big)$$ where $\varphi$ and
$\psi$ are given from (1).
 \end{thm}

\begin{thm}
Let $M$ be a two-dimensional spacetime with compact Cauchy
surfaces. Then, there exists two-dimensional spacetime
$\overline{M}$ with non-compact Cauchy surfaces of which the group
of conformal diffeomorphisms contains $D$ as a subgroup such that
$Con(M)$ is isomorphic to $A/D$ where $A$ is the normalizer of $D$
in $Aut(\overline{M})$. Conversely, given two-dimensional
spacetime $\overline{M}$ with non-compact Cauchy surfaces, if the
group of conformal diffeomorphisms of $\overline{M}$ contains $D$
as a subgroup, then there exists two-dimensional spacetime $M$
with compact Cauchy surfaces such that $Con(M)$ is isomorphic to
$A/D$.

\end{thm}

\begin{thm}
Let $M$ be a two-dimensional spacetime with compact Cauchy
surfaces. Then, $Con(M)$ is isomorphic to a subgroup of $Con(E)$.
\end{thm}

\section{Acknowledgement}

This research was supported by Basic Science Research Program
through the National Research Foundation of Korea(NRF) funded by
the Ministry of Education, Science and Technology(2014044991).

\end{document}